\documentclass[10pt]{IEEEtran}
\IEEEoverridecommandlockouts
\usepackage{graphicx}
\usepackage{amsmath}
\usepackage{amsfonts}
\DeclareMathOperator{\rank}{rank}
\DeclareMathOperator{\eig}{eig}
\DeclareMathOperator{\Tr}{Tr}
\DeclareMathOperator{\diag}{diag}
\DeclareMathOperator{\sgn}{sgn}
\newtheorem{theorem}{Theorem}

\newtheorem{lemma}{Lemma}

\begin{document}

\title{\LARGE \bf Input Design for System Identification via Convex Relaxation}
\author{Ian R. Manchester\\ \ \\Electrical Engineering and Computer Science\\Massachusetts Institute of Technology\\irm@mit.edu}

\maketitle

\begin{abstract}
This paper proposes a new framework for the optimization of excitation inputs for system identification. The optimization problem considered is to maximize a reduced Fisher information matrix in any of the classical D-, E-, or A-optimal senses. In contrast to the majority of published work on this topic, we consider the problem in the time domain and subject to constraints on the amplitude of the input signal. This optimization problem is nonconvex. The main result of the paper is a convex relaxation that
gives an upper bound accurate to within $2/\pi$ of the true maximum. A randomized algorithm is presented for finding a feasible solution which, in a certain sense is expected to be at least $2/\pi$ as informative as the globally optimal input signal. In the case of a single constraint on input power, the proposed approach recovers the true global optimum exactly. Extensions to situations with both power and amplitude constraints on both inputs and outputs are given. A simple simulation example illustrates the technique.
\end{abstract}

\section{Introduction}

System identification is the process of computing a compact mathematical model of a real-world system based on experimental input-output data. The quality of the model so identified can depend a great deal on the choice of excitation input. In many practical applications it is natural to seek to extract as much relevant information from the system as possible in minimal time, subject to certain experimental constraints.

Over several decades, a large body of literature has developed on the topic of {\em optimal input design} (see, e.g., \cite{goodwin_payne, LjungBook, Jansson05, bombois,  Rojas07, Rivera09} and references therein). Essentially the same problem is studied in the communications literature for finding test signals for channel estimation (see, e.g., \cite{Tellambura99,Li02, Fragouli03} and many others). Most channel estimation systems assume quite simple dynamic models -- either a static input-output map or an FIR filter -- however recently more dynamic models have been considered \cite{Zhang05}.

The bulk of input design methods for dynamic systems are based on the recognition that, for a linear system, the Fisher information matrix is an affine function of the input power spectrum. Imposing power constraints via Parseval's theorem, and optimizing over an affine parametrization of the input spectrum, the optimization can be posed as a semidefinite program, for which efficient computational tools are readily available \cite{Jansson05}.

In the frequency domain it is natural to consider signal power constraints. However, in many practical cases, the real constraint is the amplitude of the excitation input, not its power. In industrial processes, amplitude constraints are common, as evidenced by the success of model model predictive control. Furthermore, in the emerging area of biomedical system identification safety limits are often given as amplitude constraints (see, e.g. \cite{Kennet10, CSF_CDC}); in communication channel estimation, often a binary signal is desired. The relationship between signal phases and peak amplitude is highly non-linear and non-smooth, making optimization under amplitude constraints computationally challenging in the frequency domain. In previous work we have applied a P\'olya-like algorithm to this problem, solving first a convex optimization with power constraints followed by a sequence of smooth nonlinear optimizations \cite{Manchester09}. 

In the time domain, amplitude constraints appear naturally, and one can study linear time-varying systems and time-varying constraints, as may be appropriate to estimate intrinsic parameters of a nonlinear system via small deviations about a changing operating point. However, the resulting optimization problem is highly nonconvex: even for a system with one parameter to identify, the optimization for an input signal of length $n$ will have $2^n$ local optima. Recent work has suggested using a frequency-domain power-constrained optimization as an initial guess for a local BMI optimization algorithm \cite{Suzuki07}. Others have suggested a minimizing one-step-ahead parameter variance \cite{Lacy03a}, or optimizing the transition probabilities of a markov process for the input \cite{Brighenti09}. In this paper, we show that techniques from semidefinite relaxation of non-convex quadratic programs can be extended to the time-domain experiment design problem to find approximate solutions with a high degree of efficiency. 

\subsection{Semidefinite Relaxations of Nonconvex Quadratic Programs}

As shall be seen, the problem of maximizing an information matrix in the time domain has a structure similar to certain nonconvex quadratic programming problems. Such problems are in general NP-Hard to solve exactly, however it has been found that specific semidefinite relaxations can give upper bounds on the objective and lead to efficient randomized algorithms to find feasible suboptimal solutions with a high degree of accuracy. The breakthrough result of Goemans and Williamson \cite{Goemans95} for finding the maximum cut of a graph to within approximately 0.87 of its true optimum led to many more applications in combinatorial optimization, signal processing, operator theory, and systems analysis \cite{Wolkowicz00,Luo10, Meg01}. An essential tool in this paper will be Nesterov's extension to maximization of a positive-definite quadratic form over a hypercube with an accuracy ratio of $2/\pi$ \cite{Nesterov98, Wolkowicz00}. This family of methods can be variously interpreted as either a simple relaxation of the feasible-set, the dual of the Lagrangian dual, or optimization of the covariance of a random variable \cite{Wolkowicz00, Meg01}.

\subsection{Paper Structure}
The Structure of the paper is as follows: in Section \ref{sec:problem} we introduce the problem statement mathematically; in Section \ref{sec:convex1} we give a convex relaxation of the input design problem with input amplitude constraints, and the main theoretical results of the paper; in Section \ref{sec:convex2} we extend the solution to more general contraint types; in Section \ref{sec:examples} we give some illustrative examples; Section \ref{sec:conc} contains some brief conclusions and future directions. 

\subsection{Notation}
$S_n^+$ the cone of symmetric $n\times n$ positive-semidefinite matrices. For symmetric matrices, $X\ge Y$ means $X-Y\in S_n^+$. 
The function $\sgn(\cdot):\mathbb{R}^n\rightarrow\{-1,0,1\}^n$ computes a vector, each element of which is the sign of the corresponding element of the argument vector. The symbol $\mathbb{E}$ denotes the expectation operator. 
We make the following definition: a function $v: S_n^+ \rightarrow \mathbb{R}$ is denoted {\em nonnegative-concave} if $v(X)\ge 0 \ \forall \, X \in S_n^+$ and $v(\alpha X + (1-\alpha)Y)\ge \alpha v(X) + (1-\alpha)v(Y) \ \forall \, X,Y\in S_n^+, \, \alpha \in [0,1]$.

\section{Optimal Experiment Design}\label{sec:problem}

In a statistical experiment design, the amount of information about parameters $\theta$ contained in the observations $y$ from an experiment is measured by the Fisher information matrix $I(\theta)$, which depends on the experimental conditions. The Fisher information matrix is defined as
\begin{equation}\notag
I_\theta :=\mathbb E \left[\frac{\partial\log p(y|\theta)}{\partial\theta} \frac{\partial\log p(y|\theta)}{\partial\theta}'\right]
\end{equation}
where $p(y|\theta)$, considered as a function of $\theta$ with fixed observations $y$, is the likelihood function, and the derivitives are taken at the true value of $\theta$. The inverse $I(\theta)^{-1}$ is a lower bound on the achievable covariance matrix of an unbiased estimator \cite{Federov72}.

We consider dynamic system estimation problems, where from observations of finite-length sequences $u(t)$ and $y(t)$ one must estimate a system:
\[
y(t) = \mathcal G_\theta u(t) + \mathcal H_\theta e(t)
\]
where $\mathcal G_\theta$ and $\mathcal H_\theta$ are unknown linear maps parametrized by $\theta$ and $e(t)$ is a Gaussian white noise sequence. Note that this framework naturally allows multi-input multi-output systems via stacking inputs and outputs, as well as time-varying linear systems.

For simplicity we address here the particular case where $\mathcal G_\theta = \mathcal G(q)$ and $\mathcal H_\theta = \mathcal H(q)$ are single-input single-output finite-dimensional LTI systems given as rational functions of the shift operator $q$. The more general cases are a straightforward extension of our method. With this structure, the log-likelihood function is given by:
\begin{equation}\notag
\log p(y|\theta) = -\frac{n}{2}\log(2\pi)-\frac{n}{2}\log(\sigma_e)-\frac{1}{2\sigma_e}\sum_{t=1}^n\varepsilon(t)^2
\end{equation}
where
\begin{equation}\notag
\varepsilon(t) := \mathcal{H}_\theta(q)^{-1}[y(t)-\mathcal{G}_\theta(q)u(t)].
\end{equation}
Assuming an open-loop experiment, zero correlation between $u(t)$ and $e(t)$, and independently parametrized system and disturbance model -- i.e. $\theta = [\theta_G \ \theta_H]'$ -- then the information matrix can be decomposed as
\begin{equation}\notag
I_\theta=\begin{bmatrix}
\bar I_\theta(u)&0\\
0&\star\end{bmatrix}
\end{equation}
where $\bar I_\theta(u)$ is the block corresponding to $\theta_G$ and depends on the input, and $\star$ is the block corresponding to $\theta_H$ and depends only on the disturbance, and thus cannot be optimized by choice of input. Hence optimizing $I(\theta)$ is equivalent to optimizing the upper-left block:
\begin{equation}\notag
\bar  I_\theta(u) := \sum_{t=1}^n \left(\frac{\partial \varepsilon(t)}{\partial \theta_G}\right)\left(\frac{\partial \varepsilon(t)}{\partial \theta_G}\right)'.
\end{equation}

Now, suppose the system model has $N$ components, i.e. $\theta_G\in \mathbb R^N$, and consider
\begin{equation}\notag
\mathcal{F}_\theta(q) := -\mathcal{H}_{\theta_H}(q)^{-1}\frac{\partial \mathcal{G}_{\theta_G}(q)}{\partial \theta_G}u(t).
\end{equation}
The system $\mathcal{F}_\theta$ is a linear system with one input and $N$ outputs, each output representing the sensitivity of one of the parameters in $\theta_G$ to the choice of input. Let $\mathcal{F}_i, i = 1,2, ... N$ denote the rows of $\mathcal{F}_\theta(q)$. Note that
\begin{equation}\notag
\frac{\partial \varepsilon(t)}{\partial \theta_G} = \begin{bmatrix}\mathcal{F}_1(q)u(t)\\ \vdots \\\mathcal{F}_N(q)u(t)\end{bmatrix}
\end{equation}

Let us define a stacked control vector $u := [u(1),  u(2), ...,  u(n)]'$ and define a matrix $F_i \in \mathbb R^{n\times n}$ representing the action of each $\mathcal{F}_i$ on $u$, i.e. let $f_i(t)$ be the impulse response of $\mathcal{F}_i$, then $F_i$ is the Toeplitz matrix:
\begin{equation}\label{eqn:toep}
F_i := \begin{bmatrix} f_i(1) &0 & \hdots& 0\\
  f_i(2) & f_i(1) & \hdots & 0\\
\vdots &\vdots &\ddots &\vdots &\vdots\\
f_i(n) & f_i(n-1)  & \hdots & f_i(1)
\end{bmatrix}.
\end{equation}
Then one can represent the elements of the reduced information matrix as
\begin{equation}\notag
\bar I_\theta(u) _{i,j}=\sum_{t=1}^n (\mathcal{F}_i(q)u(t))(\mathcal{F}_j(q)u(t)) = u'F_i'F_ju
\end{equation}
Hence we have 
\begin{equation}\label{eqn:M1}
\bar  I_\theta(u) = \begin{bmatrix} u'F_1'F_1u & \hdots & u'F_1'F_Nu\\
\vdots & \ddots & \vdots\\
 u'F_N'F_1u &\hdots & u'F_N'F_Nu \end{bmatrix}
\end{equation}
as a compact expression for the reduced information matrix in terms of quadratic forms in $u$. In the general MIMO or time-varying cases, an equivalent $\bar I_\theta(u)$ could be similarly constructed by stacking control inputs and computing the time-varying equivalent of $F_i$.

Note that in most cases $\mathcal G(q)$ is nonlinearly parametrized by $\theta$, and hence the above computations depend on the true value of $\theta$. This seems somewhat paradoxical, but in most practical cases a reasonable guess for $\theta$ can be made, or multi-stage adaptive or robustified optimizations can be performed \cite{Rojas07}.

\subsection{Optimality Criteria}

The purpose of input design is to maximize, in some sense, the information matrix. In this paper, we consider maximizing an objective function of the form
$$
v(u)= J[\bar  I_\theta(u) ]
$$
where $J(\cdot):S_n^+\rightarrow \mathbb{R}$ is any nonnegative-concave function, which acts on the reduced information matrix $\bar I_\theta(u)$ generated by $u$.

In the experiment design literature the following optimization criteria are common, and all are nonnegative-concave:
\begin{itemize}
\item D-Optimality: $J_D[\bar  I_\theta(u) ] := \det[\bar  I_\theta(u) ]^{\frac{1}{n}}$,
\item E-Optimality: $J_E[\bar  I_\theta(u) ] := \min\eig[\bar  I_\theta(u) ]$,
\item A-Optimality: $J_A[\bar  I_\theta(u) ] := -\Tr[\bar  I_\theta(u) ^{-1}]$.
\end{itemize}

Note that D-Optimality is usually defined as maximizing $ \det[\bar  I_\theta(u) ]$, however this is not concave and any $u$ achieving this clearly also maximizes $ \det[\bar  I_\theta(u) ]^{\frac{1}{n}}$ which is concave. Another possible concave function with equivalent maxima would be $ \log\det[\bar  I_\theta(u) ]$.

\subsection{Constraints}
In the initial part of the paper, we will consider amplitude constraints on the input signal. These constraints may be time varying:

\[
|u(t)|\le c(t) \ \forall t= 1, 2, ... n,
\]
for some positive constraint sequence $c(t)$. More general constraints including constraints on power and output signals will be considered in Section \ref{sec:convex2}.

\section{Amplitude-Constrained Problem}\label{sec:convex1}
The time-domain amplitude-constrained input design optimization problem can be expressed like so:
\begin{equation}
    \label{eq:gen_prob}
 v^\star:=   \max_{u\in \mathbb{R}^n, |u_i|\le c_i} J\left(\begin{bmatrix}u'Q_{1,1}u&\hdots&u'Q_{1,N}u\\
      \vdots&\ddots&\vdots\\
      u'Q_{N,1}u&\hdots&u'Q_{N,N}u\end{bmatrix}\right)
\end{equation}
where each $Q_{i,j}\in\mathbb{R}^{n,n}$, each diagonal block $Q_{ii}>0$ and the matrix
\begin{equation}
  \label{eq:1}
  \begin{bmatrix}Q_{1,1}&\hdots&Q_{1,N}\\
      \vdots&\ddots&\vdots\\
      Q_{N,1}&\hdots&Q_{N,N}\end{bmatrix}\ge 0
\end{equation}

We first note that by making the substitution $U=uu'$, and from the fact that trace is cyclic, for each element of the information matrix $u'Q_{i,j}u=\Tr(uu'Q_{i,j})=\Tr(UQ_{i,j})$. In this form, the optimization problem (\ref{eq:gen_prob}) is equivalent to
\begin{equation}
    \label{eq:nonconv_opt}
   v^\star= \max_{U\in C} J\left(\begin{bmatrix}\Tr(UQ_{1,1})&\hdots&\Tr(UQ_{1,N})\\
      \vdots&\ddots&\vdots\\
      \Tr(UQ_{N,1})&\hdots&\Tr(UQ_{N,N})\end{bmatrix}\right)
\end{equation}
where the feasible set $C$ is defined as
\begin{equation}
  \label{eq:2}
  C:=\{U\in S_n^+:  U_{i,i}\le c_i^2, \rank(U)=1\}.
\end{equation}
Note that the maximization is now concave in the decision variable $U$ but the feasible set is non-convex due to the rank constraint: in general, a convex combination of two rank-one matrices can have rank two. The constraints that $U$ be positive-definite and have diagonal elements less than one are both convex.

Following \cite{Goemans95, Nesterov98}, we ``relax'' this nonconvex optimization problem by dropping the troublesome rank constraint
\begin{equation}
    \label{eq:gen_rel}
   v_R:= \max_{U\in R} J\left(\begin{bmatrix}\Tr(UQ_{1,1})&\hdots&\Tr(UQ_{1,N})\\
      \vdots&\ddots&\vdots\\
      \Tr(UQ_{N,1})&\hdots&\Tr(UQ_{N,N})\end{bmatrix}\right)
\end{equation}
where the relaxed feasible set is
\begin{equation}
  \label{eq:2}
  R:=\{U\in S_n^+:  U_{i,i}\le c_i^2\} 
\end{equation}
This optimization problem is the maximization of a concave function over the convex cone of semidefinite matrices, subject to affine constraints, and hence can be efficiently solved.

The maximization (\ref{eq:gen_rel}) has the same objective function as (\ref{eq:nonconv_opt}) but a larger feasible set, i.e. $C \subset R$. Hence the relaxed problem provides an upper bound, i.e. $v^\star \le v_R$. The main result of this section is that the upper bound given by the convex relaxation is quite tight:

\

\begin{theorem}\label{relaxation_bound}
The true optimal value $v^\star$ and the optimal value of the relaxed problem $v_R$ satisfy the following inequalities:
\begin{equation}
\frac{2}{\pi}v_{R} \le v^\star \le v_{R}.
\end{equation}
\end{theorem}

\

To prove this theorem will make use of the following theorem of Nesterov on the tightness of SDP upper bounds on nonconvex quadratic program:
\begin{theorem} \label{thm:Nesterov}\cite{Nesterov98} Let 
\begin{eqnarray}
v_{QP}(Q) &=& \max_{x\in \mathbb{R}^n, |x_i|\le c_i}x'Qx\\v_{SDP}(Q) &=& \max_{X\in R} \Tr(QX)\end{eqnarray}
 then for any $Q\in S_n^+$ $$\frac{2}{\pi}v_{SDP}(Q) \le v_{QP}(Q)  \le v_{SDP}(Q).$$
$\Box$
\end{theorem}
Note that 
$$
v_{QP}(Q) = \max_{X\in C}\Tr(QX).
$$
An essential step in the proof of Theorem \ref{thm:Nesterov} is the following statement which we will also make use of: if $x$ is a Gaussian random variable and $X$ is its covariance matrix, then
\begin{equation}\label{eqn:arcsin}
\mathbb E(\sgn x \sgn x') =\frac{2}{\pi}\arcsin(X)\ge \frac{2}{\pi} X
\end{equation}
where $\arcsin(\cdot)$ of a matrix denotes taking the $\arcsin$ elementwise.

\ 

\noindent {\em Proof of Theorem \ref{relaxation_bound}:}

Since the relaxed feasible set is strictly larger than the true feasible set it is guaranteed that $v^\star\le v_{R}$. Therefore we only need to prove that $$\frac{2}{\pi}v_{R} \le v^\star.$$

Since $J$ is a concave function over the symmetric matrices, there
exists a representation:
$$
J(U) = \min_i[h(i)+\Tr(H(i)U)]
$$
where $i$ varies over a possibly infinite set. Furthermore, since $J(U)\ge 0$ for all $U\in S_n^+$, then it follows that $H(i)\ge 0$ and $h(i)\ge 0$ for all $i$, otherwise there would exist a $U\in S_n^+$ making $h(i)+\Tr(H(i)U)$ negative for some $i$, and hence making $J(U)$ negative.

Let
$$
v_C = \max_{U\in C}\min_i[h(i)+\Tr(H(i)U)]
$$
And let $U_C$ and $i_C$ be a matrix and function index for which this
maximum is achieved.

Similarly, let
$$
v_R = \max_{U\in R}\min_i[h(i)+\Tr(H(i)U)]
$$
and let $U_R$ and $i_R$ be a matrix and function index for which this
maximum is achieved.

By Theorem \ref{thm:Nesterov}, fixing $i=i_C$,
\begin{equation}
  \label{eq:3}
  \frac{2}{\pi} \max_{U\in R}[\Tr(H(i_C)U)]\le  \max_{U\in C}[\Tr(H(i_C)U)]
\end{equation}
From $h(\cdot)\ge 0$ and $0<\frac{2}{\pi}<1$ it follows that $\frac{2}{\pi} h(i) < h(i)
\ \forall i$, and from this and (\ref{eq:3}) we have:
\begin{eqnarray}
  \label{eq:4}
  &&\frac{2}{\pi} \max_{U\in R}[h(i_C)+\Tr(H(i_C)U)]\notag\\ &&\le  \max_{U\in C}[h(i_C)+\Tr(H(i_C)U)]=v_C.
\end{eqnarray}
Now, clearly for all $U$
\begin{equation}
  \label{eq:5}
  \min_i[h(i)+\Tr(H(i)U)]\le h(i_C)+\Tr(H(i_C)U),\notag
\end{equation}
hence 
\begin{eqnarray}
  \label{eq:6}
  v_R&=&\max_{U\in R}\min_i[h(i)+\Tr(H(i)U)]\notag \\ &\le& \max_{U\in R}[h(i_C)+\Tr(H(i_C)U)]
\end{eqnarray}
and so from (\ref{eq:6}) and (\ref{eq:4}) we have
$$
\frac{2}{\pi} v_R \le v_C.
$$
This completes the proof of the Theorem.

\subsection{Finding a Feasible Solution}\label{sec:feas}

The solution of the convex relaxation (\ref{eq:gen_rel}) is an $n\times n$ matrix. To find an identification input, we need to somehow extract a vector of length $n$. The following randomized procedure is common in semidefinite relaxation and has proven to be effective in practice \cite{Luo10}.

Compute a matrix $D\ge 0$ such that $U=D'D$ Sample a vector $\xi \in \mathbb{R}^n$ with each element an independent normally distributed random variable. Compute a candidate solution 
\begin{equation}\label{eqn:random}
\hat u=\diag(c)\sgn(D'\xi),
\end{equation}
where $\diag(c)$ is a square matrix with the elements of the constraint sequence $c$ on the main diagonal, and zeros elsewhere.

Let $U_R$ be the solution of the relaxed optimization problem \eqref{eq:gen_rel}, and $\bar I_\theta(U_R)$ be the reduced information matrix with $U=U_R$, i.e.
\[
\bar I_\theta(U_R)=\begin{bmatrix}\Tr(UQ_{1,1})&\hdots&\Tr(UQ_{1,N})\\
      \vdots&\ddots&\vdots\\
      \Tr(UQ_{N,1})&\hdots&\Tr(UQ_{N,N})\end{bmatrix},
\]
then we can state the following theorem:

\

\begin{theorem}\label{feasible_bound}
The expectation of the reduced information matrix generated by (\ref{eqn:random}) satisfies the following bound
\begin{equation}\label{eqn:feasible_bound}
\mathbb E \left(\bar I (\hat u)\right) \ge \frac{2}{\pi}\bar I (U_R).
\end{equation}
\end{theorem}

\ 

Essentially, one can say that the randomized strategy is expected to give inputs at least $2/\pi$ as informative as the solution of the relaxed problem, in terms of the reduced information matrix.

\

\noindent {\em Proof of Theorem \ref{feasible_bound}:}

The proof will use on the following lemma:

\begin{lemma}\label{lem:trace}
Let 
\begin{equation}
A := \begin{bmatrix}A_{1,1}&\hdots&A_{1,N}\\
      \vdots&\ddots&\vdots\\
      A_{N,1}&\hdots&A_{N,N}\end{bmatrix}
\end{equation}
where each block $A_{i,j}$ is square of dimension $n$. Suppose $A \ge 0$, then $A_{\Tr } \ge 0$ where
\begin{equation}
A_{\Tr } := \begin{bmatrix}\Tr(A_{1,1})&\hdots&\Tr(A_{1,N})\\
      \vdots&\ddots&\vdots\\
      \Tr(A_{N,1})&\hdots&\Tr(A_{N,N})\end{bmatrix}
\end{equation}
$\Box$
\end{lemma}
\begin{proof}
Define vectors $e_i(x)\in \mathbb R^n$ for $i= 1, 2, .. N$ to have $x$ at the $i^{th}$ element and zeros at every other element, e.g. $e_1(x)=[x \ 0 \ 0 \ \cdots \ 0]'$, consider $z\in \mathbb R^N$ and consider $\bar e_k(z)$ to be the vector $[e_k(z_1)' \  e_k(z_2)' \ \cdots \ e_k(z_N)']$. Then we have
\[
\bar e_k(z)'A\bar e_k(z) = z'A_kz
\]
where $A_k$ is formed by selecting out the $(k, k)$ element of each block $A_{i,j}$, denoted $A_{i,j}(k,k)$:
\[
A_k :=\begin{bmatrix}A_{1,1}(k,k)&\hdots&A_{1,N}(k,k)\\
      \vdots&\ddots&\vdots\\
      A_{N,1}(k,k)&\hdots&A_{N,N}(k,k)\end{bmatrix}
\]
Since $A \ge 0$, $\bar e_k(z)'A\bar e_k(z) \ge 0$ for any $z$, so $A_k\ge 0$ for each $k$. Now
\[
A_{\Tr} = A_1+A_2+...A_n
\]
which is clearly positive semidefinite since each $A_k\ge 0$.
\end{proof}

\ 

To prove Theorem \ref{feasible_bound}, we must show that
\[
\mathbb E \left(\bar I (\hat u)\right) \ge \frac{2}{\pi}\bar I (U_R) 
\]

Now,
\begin{equation}\label{eqn:Mblock}
\mathbb E \left(\bar I (\hat u)\right) - \frac{2}{\pi}\bar I (U_R) = \begin{bmatrix}\Tr(M_{1,1})&\hdots&\Tr(M_{1,N})\\
      \vdots&\ddots&\vdots\\
      \Tr(M_{N,1})&\hdots&\Tr(M_{N,N})\end{bmatrix}
\end{equation}
where
$$
M_{i,j} = Q_{i,j}\left(\mathbb E(\hat u \hat u') - \frac{2}{\pi}U_R\right)
$$
Now, since $\hat u = \sgn (\xi)$ with $\xi$ a Gaussian random variable, it follows from (\ref{eqn:arcsin}) that
\begin{equation}\label{eqn:Urdiff}
\mathbb E\left(\hat u \hat u'\right) - \frac{2}{\pi}U_R = \frac{2}{\pi}(\arcsin(U_R)-U_R)\ge 0.
\end{equation}
Therefore one can define $E\ge 0$ such that
$$
EE'=\mathbb E\left(\hat u \hat u'\right) - \frac{2}{\pi}U_R
$$
and by cyclic trace an equivalent formulation of $\mathbb E \left(\bar I (\hat u)\right) - \frac{2}{\pi}\bar I (U_R) $ is given by (\ref{eqn:Mblock}) with
$$
M_{i,j} = E'Q_{i,j}E.
$$
With this substitution, and the fact that the block matrix $Q$ is positive semidefinite, it is clear that the block matrix $M$ is semidefinite. Now, from Lemma \ref{lem:trace} and (\ref{eqn:Mblock}) it follows that
$$
\mathbb E \left(\bar I (\hat u)\right) \ge \frac{2}{\pi}\bar I (U_R).
$$
This completes the proof of the Theorem.

\section{Relaxations for Problems with Power and Output Constraints}\label{sec:convex2}

In this section we extend the above relaxation approach to a broader variety of signal constraints.

\subsection{Constraints on Input Power}

First we examine a single constraint on the input power:
\[
\|u\|_2^2 \le p_u
\]
 and show in this special case that the relaxation approach finds a global optimum.

The power-constrained input design problem is to find
\begin{equation}
    \label{eq:gen_prob_power}
 v^\star:=   \max_{u\in \mathbb{R}^n, \|u\|_2^2 \le p_u} J\left(\begin{bmatrix}u'Q_{1,1}u&\hdots&u'Q_{1,N}u\\
      \vdots&\ddots&\vdots\\
      u'Q_{N,1}u&\hdots&u'Q_{N,N}u\end{bmatrix}\right).
\end{equation}

Using again the substitution $U=uu'$ we define the true constraint set, and the relaxed constraint set dropping the rank constraint:
\begin{eqnarray}
C_P &=& \{U\in S_n^+: \Tr U \le p_u, \rank(U)=1\},\notag \\
R_P &=& \{U\in S_n^+: \Tr U \le  p_u\},\notag
\end{eqnarray}
and then (\ref{eq:gen_prob_power}) is equivalent to:
\begin{equation}
    \label{eq:nonconv_opt_power}
   v^\star= \max_{U\in C_P} J\left(\begin{bmatrix}\Tr(UQ_{1,1})&\hdots&\Tr(UQ_{1,N})\\
      \vdots&\ddots&\vdots\\
      \Tr(UQ_{N,1})&\hdots&\Tr(UQ_{N,N})\end{bmatrix}\right),
\end{equation}
and the relaxed problem is
\begin{equation}
    \label{eq:relax_power}
   v_R:= \max_{U\in R_P} J\left(\begin{bmatrix}\Tr(UQ_{1,1})&\hdots&\Tr(UQ_{1,N})\\
      \vdots&\ddots&\vdots\\
      \Tr(UQ_{N,1})&\hdots&\Tr(UQ_{N,N})\end{bmatrix}\right).
\end{equation}

For this special case it happens that the relaxation actually attains the optimal value. This theorem is analogous to results on the ``hidden convexity'' of trust-region optimization problems (see, e.g., \cite{BenTal96} and many others).

\

\begin{theorem}\label{thm:power}
Let $U^\star$ be a solution of the convex optimization (\ref{eq:relax_power}). Take $\hat u$ to be the eigenvector corresponding to any nonzero eigenvalue of $U^\star$, e.g. the largest eigenvalue. Then $\hat u$ achieves the global optimum of the nonconvex optimization (\ref{eq:gen_prob_power}).
\end{theorem}

\

\noindent {\em Proof of Theorem \ref{thm:power}:}

Consider the ``true'' optimization over $C$:
\[
U^\star = \arg\max_{U\in C} \min_i[h(i)+\Tr(H(i)U)]
\]
Fix $i_C$ to be the index at which this optimum exists, then we can consider equivalently
\[
\max_{U\in C}\Tr(H(i_C)U) = \max_{u\in \mathbb R^n, u'u 
\le 1} u'H(i_C)u
\]
Since $H(i_C)\ge 0$ we can take the eigendecomposition of $H(i_C)=V\Lambda V'$ with $VV'=I$ is orthogonal and consider the change of variables $z = V'u$. Then we have the following equivalent optimization problem:
\[
\max_{z\in \mathbb R^n, z'VV'z 
\le 1} z'\Lambda z = \max_{z\in \mathbb R^n, z'z 
\le 1} z_i^2 \Lambda_{ii}.
\]
The maximum value of this optimization is the largest diagonal element of $\Lambda$, i.e. the largest eigenvalue of $H(i_C)$. Now, consider the same optimization over the relaxed feasible set under the change of variables $Z = V'UV$:
\begin{eqnarray}
\max_{Z \in R} \Tr(H(i_C)VZV')] &=& \max_{Z \in R} \Tr(V'H(i_C)VZ)]\notag \\
& =& \max_{Z \in R} \Tr(\Lambda Z)]\notag \\
& =& \max_{Z \in R} \sum_i \Lambda_{i,i}Z_{i,i}\label{eqn:Zopt}
\end{eqnarray}
It is clear that optimizing solution $Z$ is a diagonal matrix with rank at most $k$, where $k$ is the multiplicity of the largest eigenvalue of $\Lambda$ (and hence $H(i_C)$),  and this eigenvalue is the maximum value of the optimization (\ref{eqn:Zopt}). It follows that if $u$ is any eigenvector $H(i_C)$ corresponding to this eigenvalue, then
\begin{eqnarray}
u'H(i_C)u &=& \max_{U\in C_P}\Tr(H(i_C)U)\notag \\ &=&  \max_{U\in R_P}\Tr(H(i_C)U) \notag
\end{eqnarray}
and any solution of the relaxed problem 
\[
U^\star = \arg\max_{U \in R} \min_i[h(i)+\Tr(H(i)U)]
\]
is a matrix of rank at most $k$. Take an orthogonal eigendecomposition of $U$ and any eigenvector corresponding to a nonzero eigenvalue is a solution of the optimization problem (\ref{eq:gen_prob_power}). 
This completes the proof of the Theorem

\subsection{General Constraints on Input and Output Signals}

In this section we consider optimizing the information matrix subject to more general power and asymmetric amplitude constraints on both input and output. It will be shown that similar relaxation methods can be directly applied.

Suppose we have an input and output power constraints:
\[
\|u\|_2^2\le p_u, \ \|y\|_2^2 \le p_y
\]
as well as possibly asymmetric and time-varying constraints on the input and output:
\[
u_{\min ,i}\le u_i \le u_{\max ,i},  \ y_{\min ,i}\le y_i \le y_{\max ,i},
\]
for all $i = 1, 2, ..., n$.

Let $G_i$ be the $i^{th}$ row of $G$, the Toeplitz matrix representing the system $u\rightarrow y$, generated similarly to \ref{eqn:toep} with the impulse response of $\mathcal{G}$. Then with decision variables $U\in S_n^+$ and $\bar u \in \mathbb R^N$ the above constraints can be put in matrix form:
\begin{eqnarray}
\Tr(U)&\le & p_u,\notag\\
\Tr(UG'G)&\le & p_y,\notag\\
U_{ii}-\bar u_iu_{\max i}-\bar u_iu_{\min , i}&\le &-u_{\max , i}u_{\min , i},\notag\\
\Tr(UG_i'G_i)-G_i\bar uy_{\max , i}-G_i\bar uy_{\min , i}&\le &-y_{\max , i}y_{\min , i},\notag\\
U &=& \bar u \bar u'\notag.
\end{eqnarray}
where the amplitude constraints are enforced for all $i = 1, 2, ..., n$. In this case, the constraint $U = \bar u \bar u'$ is nonconvex, however it can be relaxed to the convex constraint $U \ge \bar u \bar u'$, which can be represented in Schur complement form
\[
\begin{bmatrix}
U&\bar u\\ \bar u' &1\end{bmatrix}\ge 0.
\]
With this relaxation, we again have a convex constraint set in the variables $U$ and $\bar u$.

Note that the output constraints here are for the noise-free response of the nominal model. If output constraints are strict, a conservative bound may be appropriate to account for disturbances and model uncertainty.

Having found a feasible solution of the relaxed problem, a candidate input can be chosen via a randomized strategy:
\[
\hat u = \bar u +\alpha D'\xi
\]
where $D'D=U-\bar u\bar u'$, $\xi\in \mathbb R^n$ is sampled from a normal distribution, and $\alpha\ge 0$ is a scaling parameter chosen as large as possible such that $\hat u$ satisfies all the constraints. This randomized solution could be further improved using a local optimization as in \cite{Suzuki07}. 

One might ask if uniform approximation accuracy bounds can be found as in the case of input amplitude constraints. This is unlikely, since it is known that for semidefinite relaxations of quadratic maximization subject to multiple ellipsoidal constraints, even in the case where the ellipsoids have a common center, the quality of the bound degrades logarithmically in the number of constraints \cite{Nemirovski99}. In the general case described in this section there are $2n+2$ constraints, which will be large for typical input design problems, and the provable optimality bounds are not very optimistic. However, the author has found this method to be effective in practice. This will be further explored in future publications.

\section{Illustrative Examples}\label{sec:examples}

In this section we show some results on a simple illustrative example:
\begin{equation}\label{eqn:example}
\mathcal G(q) = \frac{b}{q^2+a_1q+a_2}, \ \mathcal H = 1,
\end{equation}
with $b=0.1, a_1 = -1.8, a_2 = 0.9$. The step response of this system is shown in Figure \ref{fig:step}. We applied the proposed algorithm with a constraint that $|u(t)|\le 1$ for all $t$ and a time length of 100 samples. Note that experiment length is quite short in comparison to the transient dynamics of the system, so asymptotic arguments used in frequency-domain input design may not apply. The objective function $J(\cdot) := \det (\cdot)^{1/n}$ was used. The upper bound on the objective computed via the relaxation is 1.82$\times 10^4$.

\begin{figure}
\begin{center}
\includegraphics[width=0.8\columnwidth]{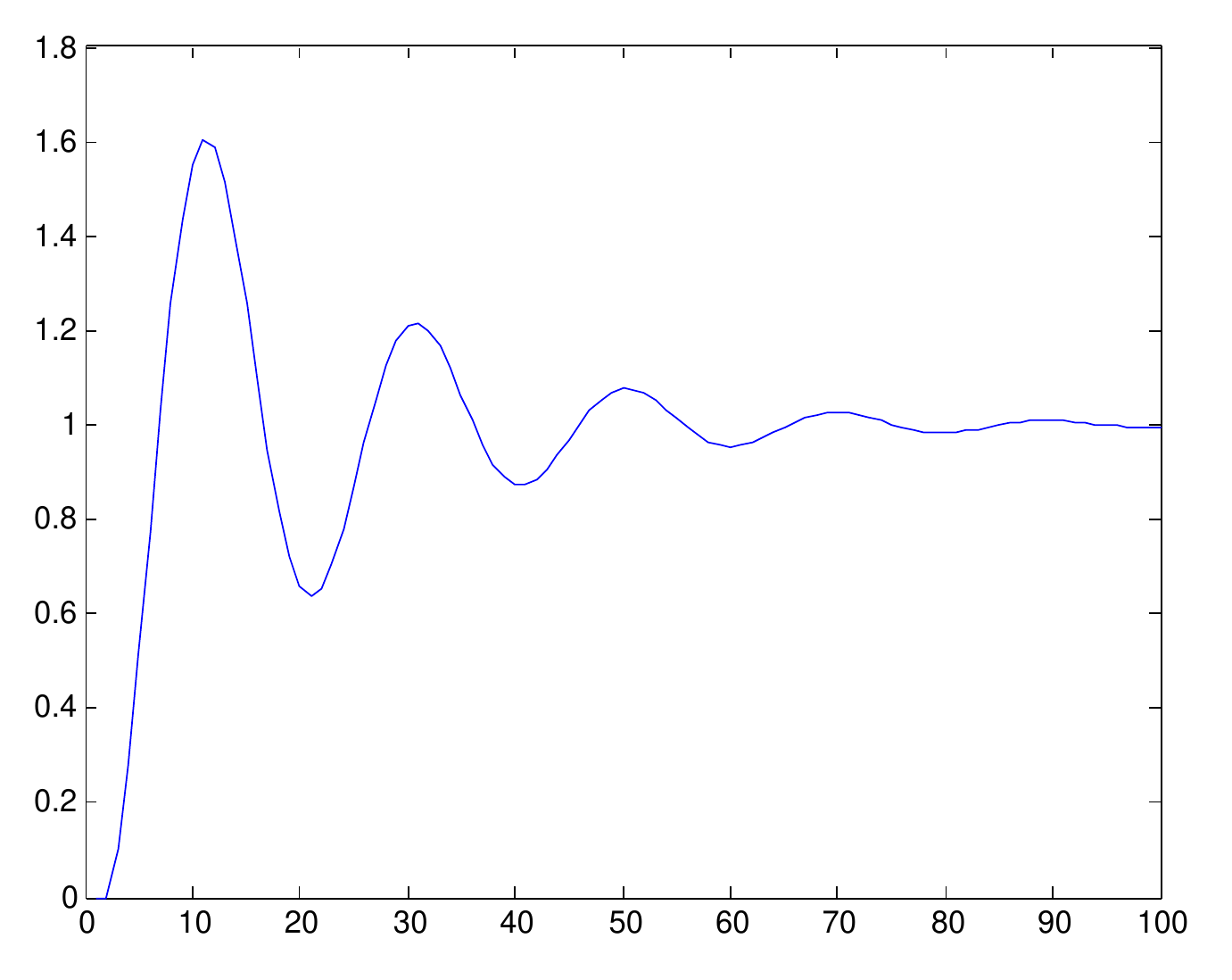}
\caption{Step response of the example system (\ref{eqn:example}).}
\label{fig:step}
\end{center}
\end{figure}

To illustrate that the randomized search for a feasible solution (see Section \ref{sec:feas}) is in some sense ``intelligent'', we generated 50,000 candidate sequences using (\ref{eqn:random}), and computed the information matrix  and the resulting objective  value. The same process was performed for a purely random binary sequence, i.e. $\sgn(\xi)$ where $\xi$ is a Gaussian white noise sequence. Figure \ref{fig:hist} shows a histogram of the ratios of objective functions these inputs to the upper bound $v_R$ computed via the relaxation. It can be seen that almost all signals generated using the proposed approach are of high quality and one could sample far fewer random candidates, although the testing candidates is not computationally expensive. In contrast, purely random binary sequences give substantially worse results. It is clear that the relaxation-based strategy biases the randomization heavily towards good inputs.

The best value achieved for the objective was 1.54$\times 10^4$. The approximation ratio of the feasible solution to the upper bound is 0.85, significantly better than the $2/\pi \approx 0.64$ which is guaranteed by Theorem \ref{relaxation_bound}. Note that we do not know what the true optimal value $v^\star$ is, only that it is between best feasible input found and the upper bound $v_R$. Therefore the feasible input we have found may in fact be closer to the global optimum than the ratio 0.85 suggests.

The response of the system to the best input is plotted in Figure \ref{fig:example}. The usefulness of the input design procedure was evaluated by comparing it to a pseudo-random binary sequence (PRBS) with the same amplitude constraints -- a frequently-used input pattern for system identification \cite{LjungBook}. A zero-mean Gaussian noise of variance 0.01 was added to $y$, and the output-error method in MATLAB System Identification Toolbox was applied. This test was performed for 500 times for each input signal with different noise patterns. Statistics of the resulting estimations are shown in Table \ref{error_tab}. It is clear that substantially better parameter estimates are found with the proposed relaxation method.

\begin{figure}
\begin{center}
\includegraphics[width=0.85\columnwidth]{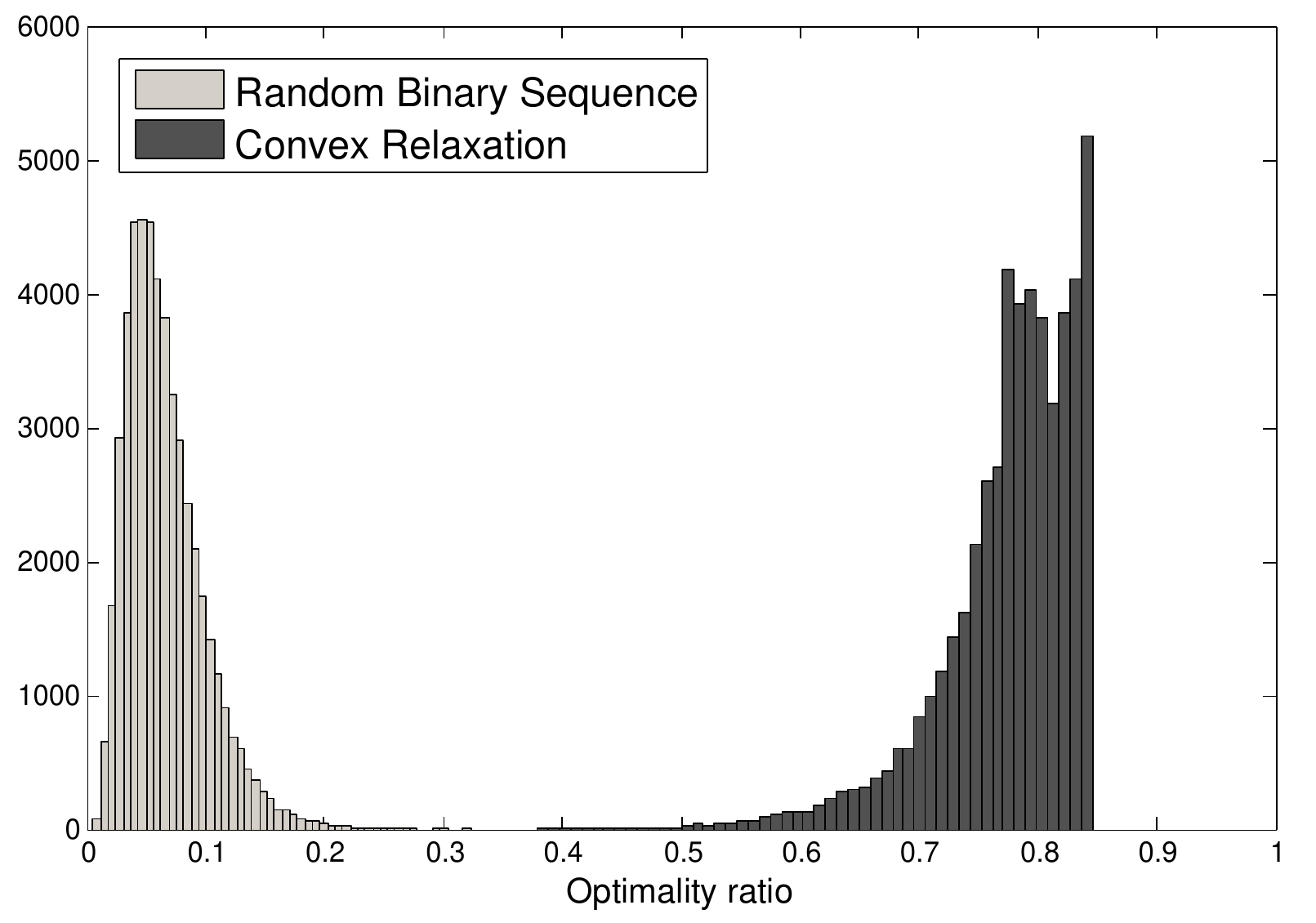}
\caption{Histogram of estimated optimality ratio of random candidate signals generated with (\ref{eqn:random}) as well as purely random binary sequences. A value of one would prove global optimality.}
\label{fig:hist}
\end{center}
\end{figure}

\begin{figure}
\begin{center}
\includegraphics[width=0.85\columnwidth]{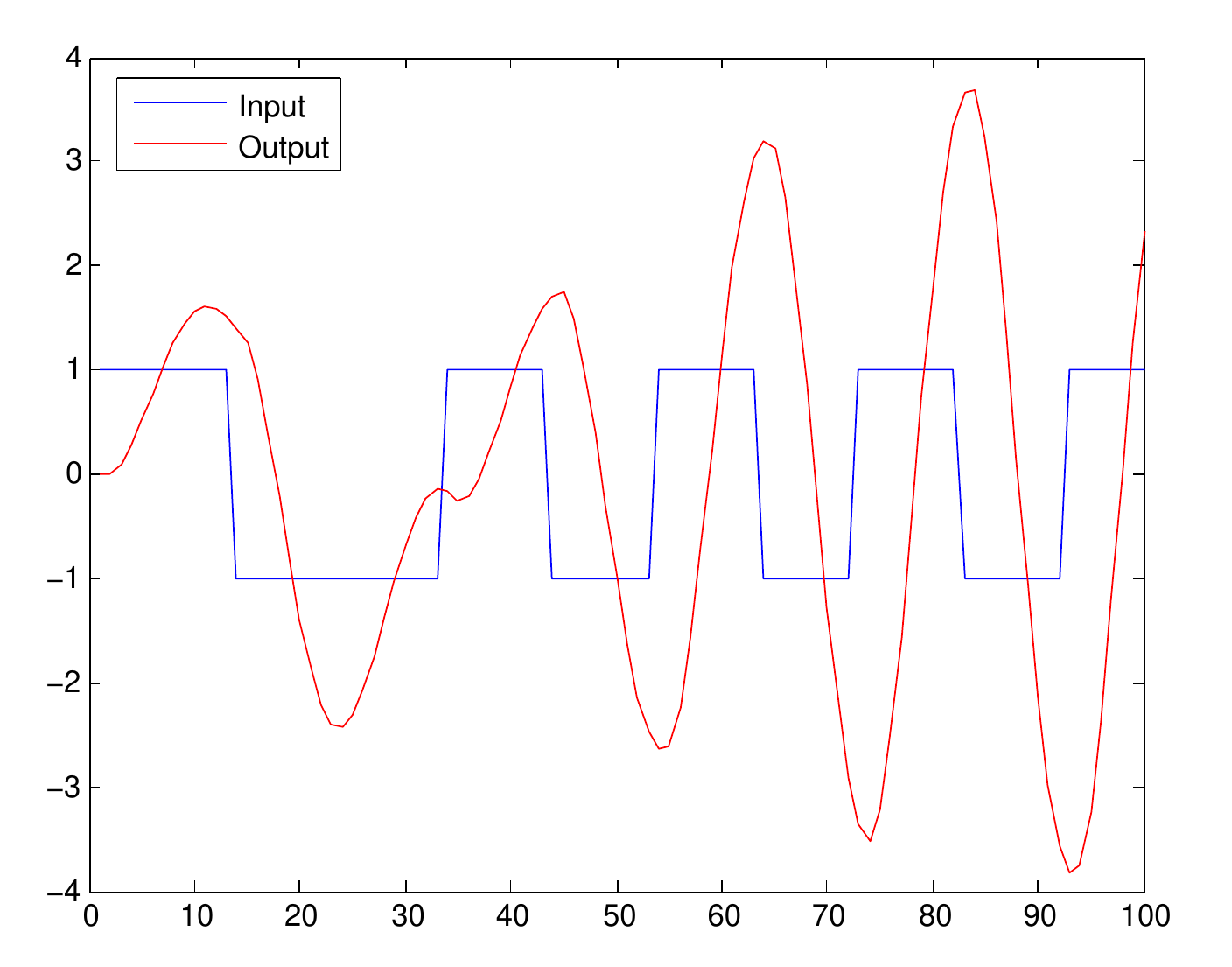}
\caption{Response of the example system (\ref{eqn:example}) to the input signal found via the proposed method.}
\label{fig:example}
\end{center}
\end{figure}

\begin{table}
\begin{center}
\begin{tabular}{| l || c | c | c | }
  \hline	
  & $a_1$ & $a_2$ & $b$\\
\hline\hline
  True & -1.8 & 0.9 & 0.1 \\ \hline
  Mean (Opt) & -1.800 &0.899 & 0.10 \\
  Std. Dev. (Opt) & 1.7$\times 10^{-3}$ & 1.7$\times 10^{-3}$  &1.1$\times 10^{-3}$ \\ \hline
  Mean (PRBS) & -1.499 & 0.810 &0.0234 \\
  Std. Dev. (PRBS) & 0.789 & 0.383 &0.044 \\
  \hline 
\end{tabular}
\caption{Parameter estimation comparison for the proposed optimization method and pseudo-random binary sequence (PRBS) on example system (\ref{eqn:example}).}
\label{error_tab}
\end{center}
\end{table}

\section{Conclusions}\label{sec:conc}

We have proposed a new framework for design of input signals for system identification in the time domain. There are many potential applications in industrial processes, biomedical modelling, communication systems, etc. The time-domain input design problem is highly nonconvex. We propose using a convex relaxation based on semidefinite programming. 

For the case of a constraint on input amplitude, we have proven that the relaxation provides an upper bound which is accurate to within $2/\pi$. Furthermore we have provided a randomized strategy for finding a feasible solution which has a similar bound in terms of expected value of a reduced information matrix. For the case of a power constraint on the input, our method provides the true global optimum. It can also be extended in a natural way to constraints on the output, as well as multi-input multi-output and time-varying systems.

A simple example illustrates the utility of the method, with the proposed strategy finding inputs which are far better than random binary sequences, resulting in substantial improvements in parameter estimates.

As a direction for future work, it is well-known that a model which is optimal in terms of its parameter estimates may not be best for describing the response of the system under feedback control (see, e.g., \cite{Gevers96, bombois}) or long term open-loop simulation (see, e.g., \cite{Farina10, Tobenkin10}). It will be interesting to explore the application of the proposed method to problems of identification for control and simulation.

\section{Acknowledgments}
The author is indebted to Alexandre Megretski for many enlightening discussions.

\bibliographystyle{IEEEtran}
\bibliography{sysid}

\end{document}